\documentclass[11pt]{amsart}

\usepackage{amsthm}
\usepackage{amssymb}
\usepackage{amsrefs}
\usepackage{enumitem}
\usepackage{graphicx}
\usepackage[colorlinks=true]{hyperref}

\theoremstyle{plain}
\newtheorem{theorem}{Theorem}[section]
\newtheorem{nontheorem}[theorem]{Non-theorem}
\newtheorem{proposition}[theorem]{Proposition}
\newtheorem{lemma}[theorem]{Lemma}
\newtheorem{corollary}[theorem]{Corollary}
\newtheorem{conjecture}[theorem]{Conjecture}
\newtheorem{claim}{Claim}[theorem]

\theoremstyle{definition}
\newtheorem{definition}[theorem]{Definition}

\newcommand{\dash}{\nobreakdash-\hspace{0mm}}

\title[Reduced clique graphs]{Reduced clique graphs: a correction to ``Chordal graphs and their clique graphs"}

\author[Mayhew]{Dillon Mayhew}
\address{School of Computing,
University of Leeds,
United Kingdom}
\email{d.mayhew@leeds.ac.uk}

\author[Probert]{Andrew Probert}
\email{andrew.mprobert@gmail.com}

\begin{document}

\begin{abstract}
Galinier, Habib, and Paul introduced the reduced clique graph of a chordal graph $G$.
The nodes of the reduced clique graph are the maximal cliques of $G$, and two nodes are joined by an edge if and only if they form a non-disjoint separating pair of cliques in $G$.
In this case the weight of the edge is the size of the intersection of the two cliques.
A clique tree of $G$ is a tree with the maximal cliques of $G$ as its nodes, where for any $v\in V(G)$, the subgraph induced by the nodes containing $v$ is connected.
Galinier et al.\ prove that a spanning tree of the reduced clique graph is a clique tree if and only if it has maximum weight, but their proof contains an error.
We explain and correct this error.

In addition, we initiate a study of the structure of reduced clique graphs by proving that they cannot contain any induced cycle of length five (although they may contain induced cycles of length three or any even integer greater than two).
We show that no cycle of length four or more is isomorphic to a reduced clique graph.
We prove that the class of clique graphs of chordal graphs is not comparable to the class of reduced clique graphs of chordal graphs by providing examples that are in each of these classes without being in the other.
\end{abstract}

\maketitle

\section{Introduction}

We consider only simple graphs.
A \emph{chord} of a cycle is an edge that joins two vertices of the cycle without being in the cycle itself.
A graph is \emph{chordal} if any cycle with at least four vertices has a chord.
A \emph{clique} is a set of pairwise adjacent vertices.
If $S$ is a set of vertices and $P$ is a path, then $P$ is \emph{$S$\dash avoiding} if no internal vertex of $P$ is in $S$.
Assuming that $a$ and $b$ are distinct vertices, an \emph{$ab$\dash separator} is a set $S$ of vertices not containing either $a$ or $b$ such that there is no $S$\dash avoiding path from $a$ to $b$.
If, in addition, $S$ does not properly contain an $ab$\dash separator then it is a \emph{minimal $ab$\dash separator}.

If $G$ is a chordal graph, then $C(G)$ is the corresponding \emph{clique graph} (also known as the \emph{clique intersection graph}).
The vertices of $C(G)$ are the maximal cliques of $G$, and two maximal cliques are adjacent in $C(G)$ if and only if they have a non-empty intersection.
The vertices of the \emph{reduced clique graph}, $C_{R}(G)$, are again the maximal cliques of $G$, but $C$ and $C'$ are adjacent in $C_{R}(G)$ if and only if $C\cap C'\ne \emptyset$ and $C$ and $C'$ form a \emph{separating pair}: that is, there is no $(C\cap C')$\dash avoiding path from a vertex in $C-C'$ to a vertex in $C'-C$.
Note that the vertices of $C_{R}(G)$ are identical to the vertices of $C(G)$, and every edge of $C_{R}(G)$ is an edge of $C(G)$.

\begin{figure}[htb]
\centering
\includegraphics{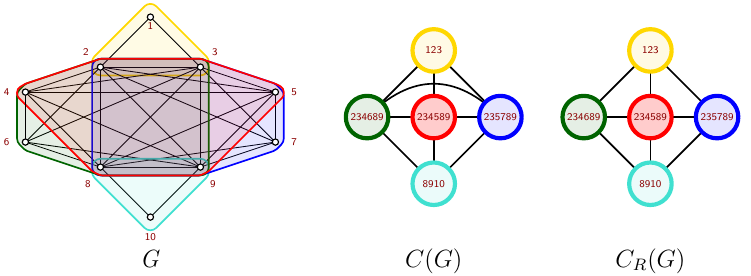}
\caption{A chordal graph, its clique graph, and its reduced clique graph.}
\label{Fig2}
\end{figure}

The reduced clique graph was introduced in~\cite{GHP95} (where it is called a clique graph) and studied further in~\cites{HS12, HL09, MUU10, McK11}.

Let $G$ be a graph, and let $T$ be a tree whose vertices are the maximal cliques of $G$.
If, for every $v\in V(G)$, the maximal cliques of $G$ that contain $v$ induce a connected subgraph of $T$, then $T$ is a \emph{clique tree}.
Clique trees were introduced by Gavril~\cite{Gavril}, who proved that a graph has a clique tree exactly when it is chordal.

We weight each edge of $C_{R}(G)$ as follows: the edge joining cliques $C$ and $C'$ is weighted with $|C\cap C'|$.
The following result is~\cite{GHP95}*{Theorem 6}.

\begin{theorem}
\label{maintheorem}
Let $G$ be a connected chordal graph.
Let $T$ be a spanning tree of $C_{R}(G)$.
Then $T$ is a clique tree if and only if it is a maximum-weight spanning tree.
\end{theorem}

Although the statement of Theorem~\ref{maintheorem} is correct, the proof in~\cite{GHP95}*{Theorem 6} is not.
The issue arises in the proof that a maximum-weight spanning tree must be a clique tree.
We illustrate the error by using the same argument to prove a false statement.

\begin{nontheorem}
\label{nontheorem}
Let $G$ be a chordal graph.
Let $C_{0},C_{1},\ldots, C_{n}$ be the sequence of maximal cliques in a path of $C_{R}(G)$ where $n>1$.
Assume that there is a vertex $v$ of $G$ such that $v$ is in $C_{0}\cap C_{n}$, but in none of the cliques $C_{1},\ldots, C_{n-1}$.
Then $C_{0}$ and $C_{n}$ are adjacent in $C_{R}(G)$.
\end{nontheorem}

\begin{proof}[Non-proof]
Consider the subgraph $G'$ of $G$ induced by $C_{0}\cup C_{1}\cup\cdots \cup C_{n}$.
Thus $G'$ is chordal.
From~\cite{Rose}*{Corollary 2} we see that either $v$ is a \emph{simplicial vertex} (meaning that the neighbours of $v$ in $G'$ form a clique), or there is a pair, $a$, $b$, of vertices such that $v$ belongs to a minimal $ab$\dash separator of $G'$.
In the former case $v$ is in a unique maximal clique of $G'$ (\cite{BP93}*{Theorem 3.1}).
But $C_{0}$ and $C_{n}$ are distinct maximal cliques of $G'$ that contain $v$.
Therefore we can let $S$ be a minimal $ab$\dash separator of $G'$, where $v$ is in $S$.
The proof of~\cite{Buneman}*{Lemma 2.3} shows that there are two distinct maximal cliques, $D_{a}$ and $D_{b}$, of $G'$ such that $D_{a}$ and $D_{b}$ properly contain $S$, and $D_{a}-S$ is in the same connected component of $G'-S$ as $a$, while $D_{b}-S$ is in the same component as $b$.
Thus $D_{a}$ and $D_{b}$ are maximal cliques of $G'$ that contain $v$.
But the only maximal cliques of $G'$ that contain $v$ are $C_{0}$ and $C_{n}$.
Therefore we can assume without loss of generality that $D_{a}=C_{0}$ and $D_{b}=C_{n}$.
Any path from a vertex of $C_{0}-C_{n}$ to a vertex of $C_{n}-C_{0}$ must contain a vertex in $S=D_{a}\cap D_{b}=C_{0}\cap C_{n}$.
Therefore $C_{0}$ and $C_{n}$ form a non-disjoint separating pair, so $C_{0}$ and $C_{n}$ are adjacent in $C_{R}(G)$, as claimed.
\end{proof}

We can see that this non-theorem is, indeed, not a theorem by examining Figure~\ref{Fig2}.
Set $C_{0}$, $C_{1}$, and $C_{2}$ to be the maximal cliques $\{2,3,4,6,8,9\}$, $\{1,2,3\}$, and $\{2,3,5,7,8,9\}$, respectively.
Thus $C_{0}, C_{1}, C_{2}$ is the vertex sequence of a path in $C_{R}(G)$.
The vertex $8$ is in $C_{0}\cap C_{2}$, but not in $C_{1}$.
However $C_{0}$ and $C_{2}$ are not adjacent in $C_{R}(G)$.
The error in the proof lies in the claim that ``the only maximal cliques of $G'$ that contain $v$ are $C_{0}$ and $C_{n}$".
This need not be true.
Indeed, $\{2,3,4,5,8,9\}$ is a maximal clique in the subgraph induced by $C_{0}\cup C_{1}\cup C_{2}$, and it contains $8$, but it is not equal to either $C_{0}$ or $C_{2}$.
Exactly the same error appears in the proof of~\cite{GHP95}*{Theorem 6}.
Nonetheless, Theorem~\ref{maintheorem} is true, and we prove it in the next section.

\section{Reduced clique graphs and clique trees}

In~\cite{MR4642470} we will apply our main theorem to some matroid problems.
For these purposes we would like to extend its scope somewhat.
Instead of weighting the edges of $C_{R}(G)$ with sizes of  intersections, we consider more general weightings.

\begin{definition}
Let $G$ be a chordal graph.
We consider a function $\sigma$ which takes
\[
\{\emptyset\}\cup\{C\cap C'\colon C,C\ \text{are distinct maximal cliques of}\ G\}
\]
to non-negative integers.
We insist that $\sigma(\emptyset)=0$ and if $X$ and $X'$ are in the domain of $\sigma$ and $X\subset X'$, then $\sigma(X)<\sigma(X')$.
In such a case the function $\sigma$ is a \emph{legitimate weighting} of $G$.
We weight an edge between $C$ and $C'$ with $\sigma(C\cap C')$.
\end{definition}

\begin{theorem}
\label{maintheorem2}
Let $G$ be a connected chordal graph and let
$\sigma$ be a legitimate weighting of $G$.
Every clique tree is a spanning tree of $C_{R}(G)$ and every edge of $C_{R}(G)$ is contained in a clique tree.
Moreover, a spanning tree of $C_{R}(G)$ is a clique tree if and only if it has maximum weight amongst all spanning trees.
\end{theorem}

Note that the function that takes each intersection $C\cap C'$ to $|C\cap C'|$ is a legitimate weighting, so Theorem~\ref{maintheorem2} does indeed imply Theorem~\ref{maintheorem}.
We now start proving the intermediate results required for the proof of Theorem~\ref{maintheorem2}.

\begin{proposition}
\label{clique-sequence}
Let $G$ be a chordal graph, and let $C$ and $C'$ be maximal cliques of $G$.
Let $S$ be a set of vertices that contains $C\cap C'$.
Let $v_{0},v_{1},\ldots, v_{k}$ be the vertex sequence of $P$, a shortest-possible $S$\dash avoiding path from a vertex in $C-C'$ to a vertex in $C'-C$.
Then $(C\cap C')\cup\{v_{i},v_{i+1}\}$ is a clique for each $i=0,1,\ldots, k-1$.
\end{proposition}

\begin{proof}
If $C\cap C'=\emptyset$ then the result holds trivially, so we assume $C\cap C'$ is non-empty.
Note that every vertex in $C\cap C'$ is adjacent to $v_{0}$, and also to $v_{k}$, since these vertices are in $C-C'$ and $C'-C$.
Now the result can only fail if there is a vertex $x\in C\cap C'$ that is not adjacent to $v_{i}$ for some $i\in\{1,\ldots, k-1\}$.
Let $p$ be the largest integer such that $p<i$ and $x$ is adjacent to $v_{p}$.
Similarly, let $q$ be the smallest integer such that $q>i$ and $x$ is adjacent to $v_{q}$.
Consider the cycle obtained by adding the edges $v_{p}x$ and $v_{q}x$ to $v_{p},v_{p+1},\ldots, v_{q}$.
This cycle contains the distinct vertices $v_{p}$, $v_{i}$, $v_{q}$, and $x$, so it must contain a chord.
No chord can join two vertices in the path $P$, since $P$ is as short as possible.
Thus any chord is incident with $x$.
But $x$ is not adjacent to any of the vertices in $v_{p+1},\ldots, v_{q-1}$ by the choice of $p$ and $q$, so we have a contradiction.
\end{proof}

\begin{proposition}
\label{CRG-path}
Let $G$ be a chordal graph, and let $C$ and $C'$ be maximal cliques of $G$ where $C\cap C'\ne \emptyset$.
If $C$ and $C'$ are not adjacent in $C_{R}(G)$, then they are joined by a path of $C_{R}(G)$ with vertex sequence $C_{0},C_{1},\ldots, C_{s}$, where each $C_{i}\cap C_{i+1}$ properly contains $C\cap C'$.
\end{proposition}

\begin{proof}
Assume this fails for $C$ and $C'$, and they have been chosen so that $C\cap C'$ is as large as possible.
Let $S$ be $C\cap C'$.
Because $C$ and $C'$ are not adjacent in $C_{R}(G)$, but $S\ne \emptyset$, it follows that there is an $S$\dash avoiding path from a vertex in $C-C'$ to a vertex in $C'-C$.
Let $v_{0},v_{1},\ldots, v_{k}$ be the vertex sequence of such a path, where $k$ is as small as possible.
We assume $v_{0}$ is in $C-C'$ while $v_{k}$ is in $C'-C$.
We apply Proposition~\ref{clique-sequence} and for each $i=1,\ldots, k$, we let $D_{i}$ be a maximal clique of $G$ that contains $S\cup\{v_{i-1},v_{i}\}$.
Set $D_{0}$ to be $C$ and set $D_{k+1}$ to be $C'$.
Note that $D_{i}\ne D_{j}$ when $i< j$, because $v_{i-1}$ is not adjacent to $v_{j}$ by the minimality of $k$.
For each $i=0,1,\ldots, k$, the intersection of $D_{i}$ and $D_{i+1}$ contains $S$ as well as $v_{i}$.
If $D_{i}$ and $D_{i+1}$ are adjacent in $C_{R}(G)$ then we let $P_{i}$ be the path of $C_{R}(G)$ consisting of $D_{i}$, $D_{i+1}$, and the edge between them.
Otherwise $D_{i}$ and $D_{i+1}$ are not adjacent in $C_{R}(G)$ and the assumption on the cardinality of $S$ means that there is a path $P_{i}$ of $C_{R}(G)$ from $D_{i}$ to $D_{i+1}$ such that every intersection of consecutive cliques in $P_{i}$ properly contains $S\cup v_{i}$.
We concatenate the paths $P_{0}, P_{1},\ldots, P_{k}$ and obtain a walk of $C_{R}(G)$ from $C$ to $C'$.
The intersection of any two consecutive cliques in this walk properly contains $S$.
It follows that there is a path of $C_{R}(G)$ from $C$ to $C'$ with exactly the same property, and now $C$ and $C'$ fail to provide a counterexample after all.
\end{proof}

Figure~\ref{Fig3} illustrates Proposition~\ref{CRG-path}.
The intersection of cliques $C=\{1,2,3\}$ and $C'=\{3,5,7,8\}$ is $\{3\}\ne\emptyset$, but $C$ and $C'$ are not adjacent in $C_{R}(G)$.
However, there is a path between $C$ and $C'$ in $C_{R}(G)$, and the intersection of any consecutive two cliques in the path properly contains $\{3\}$.

\begin{figure}[htb]
\centering
\includegraphics{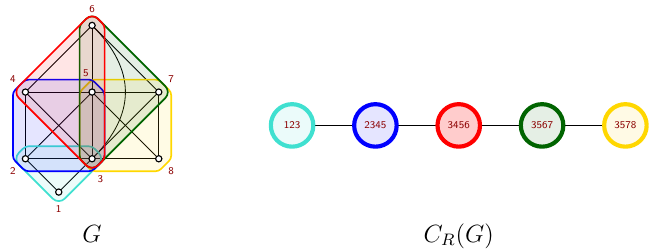}
\caption{A chordal graph and its reduced clique graph.}
\label{Fig3}
\end{figure}

\begin{proposition}
\label{clique-tree-edges}
Let $G$ be a connected chordal graph.
Let $T$ be a clique tree of $G$.
Assume that $C$ and $C'$ are maximal cliques of $G$ that are adjacent in $T$.
Then $C$ and $C'$ are adjacent in $C_{R}(G)$.
\end{proposition}

\begin{proof}
Assume $C$ and $C'$ are adjacent in $T$, but not in $C_{R}(G)$.
We partition the maximal cliques of $G$ as follows.
Let $\mathcal{U}$ be the set of maximal cliques of $G$ such that $D$ is in $\mathcal{U}$ if and only if the path of $T$ from $D$ to $C$ does not contain $C'$.
Similarly, define $\mathcal{U}'$ so that $D'$ is in $\mathcal{U}'$ if and only if the path of $T$ from $D'$ to $C'$ does not contain $C$.
Note that every maximal clique of $G$ is in exactly one of $\mathcal{U}$ or $\mathcal{U}'$, since $T$ is a tree.
Furthermore $C$ is in $\mathcal{U}$ and $C'$ is in $\mathcal{U}'$.
Let $U$ be the union of the cliques in $\mathcal{U}$, and let $U'$ be the union of the cliques in $\mathcal{U}'$.
Every vertex is in at least one maximal clique so $U\cup U'=V(G)$.
Note that $C\subseteq U$ and $C'\subseteq U'$, so neither $U$ nor $U'$ is empty.

If $U\cap U'=\emptyset$, then we choose $u\in U$ and $u'\in U'$ so that $u$ and $u'$ are adjacent in $G$.
(We are able to do so because $G$ is connected.)
The edge between $u$ and $u'$ is contained in a maximal clique.
If this maximal clique is in $\mathcal{U}$ then $u'$ is in $U\cap U'$, and if it is in $\mathcal{U}'$ then $u$ is in $U\cap U'$.
In either case we have a contradiction, so $U\cap U'\ne \emptyset$.

Choose an arbitrary vertex $v$ in $U\cap U'$.
Choose $D\in\mathcal{U}$ and $D'\in\mathcal{U}'$ such that $v$ is in $D\cap D'$.
Because $T$ is a clique tree, it follows that $v$ is contained in all the cliques belonging to the path of $T$ from $D$ to $D'$.
In particular, $v$ is contained in $C$ and $C'$.
Thus $U\cap U'\subseteq C\cap C'$ and $C\cap C'$ is non-empty.

Let $S$ be $C\cap C'$.
Since $C$ and $C'$ are not adjacent in $C_{R}(G)$, we can apply Proposition~\ref{CRG-path} and find a path $P$ of $C_{R}(G)$ from $C$ to $C'$, where the intersection of each pair of consecutive cliques in this path properly contains $S$.
Since $C$ is in $\mathcal{U}$ and $C'$ is in $\mathcal{U}'$, there is an edge of $P$ that joins a clique $D\in \mathcal{U}$ to a clique $D'\in\mathcal{U}'$.
Then $D\cap D'$ properly contains $S$, so we choose $v$ in $(D\cap D')-S$.
Again using the fact that $T$ is a clique tree, we see that the path of $T$ from $D$ to $D'$ consists of cliques that contain $v$.
In particular, $v$ is in $C\cap C' = S$, and we have a contradiction that completes the proof.
\end{proof}

It follows from Proposition~\ref{clique-tree-edges} that every clique tree of $G$ is a spanning tree of $C_{R}(G)$.

\begin{proposition}
\label{clique-tree-path}
Let $G$ be a connected chordal graph and let $\sigma$ be a legitimate weighting of $G$.
Let $T$ be a clique tree of $G$.
Let $C$ and $C'$ be maximal cliques of $G$ that are adjacent in $C(G)$ and let $P$ be the path of $T$ between $C$ and $C'$.
The weight of any edge in $P$ is at least $\sigma(C\cap C')$.
Moreover, if $C$ and $C'$ are adjacent in $C_{R}(G)$, then at least one edge in $P$ has weight equal to $\sigma(C\cap C')$.
\end{proposition}

\begin{proof}
Let $S$ be $C\cap C'$.
Let $P$ be the path of $T$ from $C$ to $C'$, and let the cliques in this path be $C_{0},C_{1},\ldots, C_{n}$, where $C_{0}=C$ and $C_{n}=C'$.
Note that $P$ is a path of $C_{R}(G)$ by Proposition~\ref{clique-tree-edges}.
Thus any two consecutive cliques in the path have a non-empty intersection.
Assume $\sigma(C_{i}\cap C_{i+1}) < \sigma (S)$ for some $i$.
If $S$ were a subset of $C_{i}\cap C_{i+1}$, then we would have $\sigma(S)\leq \sigma(C_{i}\cap C_{i+1})$ by the definition of a legitimate weighting, but this is not true.
Therefore we can choose $v$ to be a vertex in $S-(C_{i}\cap C_{i+1})$.
Now $v$ is a vertex of both $C$ and $C'$, but the path of $T$ between $C$ and $C'$ contains at least one maximal clique (either $C_{i}$ or $C_{i+1}$) that does not contain $v$.
This contradicts the fact that $T$ is a clique tree.
Therefore the weight of any edge in $P$ is at least equal to $\sigma(S)$.

Now assume that $C$ and $C'$ are adjacent in $C_{R}(G)$, so that they form a separating pair.
That is, there are distinct connected components of $G-S$ that contain, respectively, $C-S$ and $C'-S$.
There must be maximal cliques $D$ and $D'$ that are adjacent in $P$, where $D-S$ is in the same connected component of $G-S$ as $C-S$, and $D'-S$ is not in this connected component.
This means that $D\cap D'$ is contained in $S$.
Hence $\sigma(D\cap D')\leq \sigma(S)$.
The previous paragraph shows that $\sigma(D\cap D')\geq \sigma(S)$, so the result follows.
\end{proof}

The proof of the next result is a straightforward adaptation of a proof given by Blair and Peyton~\cite{BP93}*{Theorem 3.6}.

\begin{lemma}
\label{weighted-clique-graph}
Let $G$ be a connected chordal graph.
Let $\sigma$ be a legitimate weighting of $G$ and let $T$ be a spanning tree of $C(G)$.
Then $T$ is a clique tree of $G$ if and only if it is a maximum-weight spanning tree of $C(G)$.
\end{lemma}

\begin{proof}
If $T$ is a clique tree, then for any pair of maximal cliques, $C$ and $C'$, such that $C$ and $C'$ are adjacent in $C(G)$, the weight of the edge between $C$ and $C'$ is no greater than the weight of any edge in the path of $T$ between $C$ and $C'$ (Proposition~\ref{clique-tree-path}).
It immediately follows that $T$ has maximum weight.

For the other direction, we assume that $T$ is a maximum-weight spanning tree.
Because every chordal graph has a clique tree, and any clique tree is a spanning tree of $C_{R}(G)$ (and hence of $C(G)$), we can choose a clique tree $T'$ so that $T$ and $T'$ have as many edges in common as possible.
We can choose an edge in $T$ that is not in $T'$, because otherwise there is nothing left for us to prove.
So let $e$ be such an edge, and assume that $e$ joins maximal cliques $C$ and $C'$.
There are two connected components of $T\backslash e$, one containing $C$ and the other containing $C'$.
Let $P$ be the path of $T'$ from $C$ to $C'$.
We let $f$ be an edge of $P$ which joins two cliques that are not in the same component of $T\backslash e$.
Note that $f$ is an edge of $T'$, and hence an edge of $C(G)$.

If $(T-e)\cup f$ is not a spanning tree of $C(G)$, then there is a path of $T$ between the end-vertices of $f$ that does not use $e$.
But the end-vertices of $f$ are in different connected components of $T\backslash e$, so $(T-e)\cup f$ is indeed a spanning tree.
Similarly, if $(T'-f)\cup e$ is not a spanning tree, then there is a path of $T'$ between $C$ and $C'$ that does not contain $f$.
But $P$ is the unique path of $T'$ between $C$ and $C'$, and $f$ is an edge of $P$.
So $(T-e)\cup f$ and $(T'-f)\cup e$ are both spanning trees of $C(G)$.

Applying Proposition~\ref{clique-tree-path} to the clique tree $T'$ shows that the weight of $f$ is at least the weight of $e$.
Since $T$ is a maximum-weight spanning tree, and $(T-e)\cup f$ is a spanning tree it follows that the weights on $e$ and $f$ must be equal.
Let $D$ and $D'$ be the maximal cliques joined by $f$.
Any element that is in both $C$ and $C'$ must be in all the cliques in $P$, since $T'$ is a clique tree.
This shows that $C\cap C'\subseteq D\cap D'$.
If $C\cap C'$ were a proper subset of $D\cap D'$, then the definition of a legitimate weighting would mean that the weight of $e$ is strictly less than the weight of $f$, which is not true.
Therefore $C\cap C'=D\cap D'$.

We note that $(T'-f)\cup e$ cannot be a clique tree, since it has one more edge in common with $T$ than $T'$ does.
Therefore we choose a vertex $v\in V(G)$ so that the maximal cliques containing $v$ do not induce a subtree of $(T'-f)\cup e$.
Let $T''$ be the subtree of $T'$ induced by the maximal cliques containing $v$.
Then $f$ is in $T''$, or else $T''$ would be a subtree of $(T'-f)\cup e$.
This means that $v$ is in $D\cap D' = C\cap C'$.
So both $C$ and $C'$ are in $T''$, but they are not in the same component of $T''\backslash f$, because in that case $(T'-f)\cup e$ would contain a cycle.
So $e$ joins two vertices of $T''$ that are in different components of $T''\backslash f$.
Thus $(T''-f)\cup e$ is a subtree of $(T'-f)\cup e$, and we have a contradiction that completes the proof.
\end{proof}

\begin{proof}[Proof of \textup{Theorem~\ref{maintheorem2}}]
We have already noted that every clique tree is a spanning tree of $C_{R}(G)$.
Let $T$ be a clique tree of $G$.
Then $T$ is a maximum-weight spanning tree of $C(G)$ by Lemma~\ref{weighted-clique-graph}.
But every edge of $T$ is an edge of $C_{R}(G)$, by Proposition~\ref{clique-tree-edges}.
Since $C_{R}(G)$ is a subgraph of $C(G)$ it follows that $T$ is a maximum-weight spanning tree of $C_{R}(G)$.

For the other direction, we let $T$ be a maximum-weight spanning tree of $C_{R}(G)$.
We claim that $T$ is also a maximum-weight spanning tree of $C(G)$.
To prove this claim, let $e$ be an arbitrary edge of $C(G)$ that is not in $T$, let $C$ and $C'$ be the maximal cliques of $G$ joined by $e$, and let $P$ be the path of $T$ that joins $C$ and $C'$.
If $e$ is an edge of $C_{R}(G)$, then the weight of $e$ is no greater than the weight of any edge in $P$, since $T$ is a maximum-weight spanning tree of $C_{R}(G)$.
Therefore we assume that $e$ is not an edge of $C_{R}(G)$.
Now it follows from Proposition~\ref{CRG-path} and the definition of a legitimate weighting that there is a path from $C$ to $C'$ such that the edges in this path all have weight strictly greater than the weight of $e$.
From considering, say, Kruskal's algorithm, we now see that the edges in $P$ all have weight strictly greater than the weight of $e$.
In either case, the weight of $e$ does not exceed the weight of any edge in $P$.
This implies that $T$ is indeed a maximum-weight spanning tree of $C(G)$, and thus $T$ is a clique tree of $G$ by Lemma~\ref{weighted-clique-graph}.

To complete the proof, we let $e$ be an arbitrary edge of $C_{R}(G)$.
We will prove that $e$ is in a maximum-weight spanning tree of $C_{R}(G)$.
We let $C$ and $C'$ be the maximal cliques joined by $e$.
Let $T$ be an arbitrary maximum-weight spanning tree of $C_{R}(G)$, so that $T$ is a clique tree by the previous paragraph.
If $e$ is in $T$ then we have nothing left to prove, so assume that $P$ is the path of $T$ joining $C$ to $C'$, where $P$ contains more than one edge.
Proposition~\ref{clique-tree-path} shows that $P$ contains an edge, $f$, with weight equal to the weight of $e$.
Now $(T-f)\cup e$ is a maximum-weight spanning tree of $C_{R}(G)$ that contains $e$, and we are done.
\end{proof}

From the previous arguments we can deduce further additional facts, both noted in~\cite{GHP95}: any edge that is in $C(G)$ but not $C_{R}(G)$ cannot be in any maximum-weight spanning tree of $C(G)$.
Secondly, $C_{R}(G)$ is in fact the union of all clique trees of $G$.

Although the next fact is incidental to our main results here, we note it for a future application in~\cite{MR4642470}.

\begin{proposition}
\label{maxtree-leaf}
Let $G$ be a connected chordal graph, and let $T$ be a clique tree of $G$.
Let $C$ and $C'$ be adjacent in $T$ and let $S$ be $C\cap C'$.
Assume that $D$ and $D'$ are maximal cliques of $G$ and  the path of $T$ from $D$ to $D'$ contains both $C$ and $C'$.
Then $D-S$ and $D'-S$ are in different connected components of $G-S$.
\end{proposition}

\begin{proof}
Let $\mathcal{U}$ be the family of maximal cliques of $G$ such that $D$ is in $\mathcal{U}$ if and only if the path of $T$ from $D$ to $C$ does not contain $C'$.
Similarly, we let $\mathcal{U}'$ be the family of maximal cliques where $D'$ is in $\mathcal{U'}$ if and only if the path of $T$ from $D'$ to $C'$ does not contain $C$.
Note that every maximal clique of $G$ belongs to exactly one of $\mathcal{U}$ and $\mathcal{U}'$.
We are asserting that if $D\in\mathcal{U}$ and $D'\in\mathcal{U}'$, then $D-S$ and $D'-S$ are in different connected components of $G-S$.
Assume that this fails for $D$ and $D'$, where $D\cap D'$ is as large as possible.
Let $H$ be the connected component of $G-S$ that contains both $D-S$ and $D'-S$.

Let $P$ be the path of $T$ from $D$ to $D'$.
Therefore $P$ contains both $C$ and $C'$.
Let $v$ be an arbitrary vertex of $D\cap D'$.
Then $v$ is in every maximal clique that appears in $P$, since $T$ is a clique tree.
In particular, $v$ is in $C$ and $C'$.
Thus $v$ is in $S$, and this shows that $D\cap D'$ is contained in $S$.

Let $v_{0},v_{1},\ldots, v_{k}$ be the vertex sequence of a shortest-possible path of $H$ from a vertex $v_{0}\in D-S$ to a vertex $v_{k}\in D'-S$.
This is an $S$\dash avoiding path, where $S$ contains $D\cap D'$.
Thus we can apply Proposition~\ref{clique-sequence}.
For $i=1,2,\ldots, k$ we let $D_{i}$ be a maximal clique of $G$ that contains $(D\cap D')\cup\{v_{i-1},v_{i}\}$.
Let $D_{0}$ be $D$ and let $D_{k+1}$ be $D'$.
Note that each $D_{i}-S$ is contained in $H$.
This is true for $D_{0}$ and $D_{k+1}$ by definition, and every other $D_{i}$ contains the edge $v_{i-1}v_{i}$, which is in the path of $H$ from $v_{0}$ to $v_{k}$.
Since $D_{0}$ is in $\mathcal{U}$ and $D_{k+1}$ is in $\mathcal{U}'$, we can choose $i$ so that $D_{i}$ is in $\mathcal{U}$ and $D_{i+1}$ is in $\mathcal{U}'$.
The intersection of $D_{i}$ and $D_{i+1}$ is larger than $D\cap D'$, since it contains $(D\cap D')\cup v_{i}$.
As $D_{i}-S$ and $D_{i+1}-S$ are both contained in $H$ we have a contradiction to the choice of $D$ and $D'$.
\end{proof}

\section{The structure of reduced clique graphs}

Habib and Stacho comment on the possibility of investigating the structure of graphs that are isomorphic to reduced clique graphs~\cite{HS12}*{p.~714}.
In this section we make a contribution to this investigation.
We start by answering an obvious question that requires a non-trivial proof.

\begin{corollary}
\label{connected}
Let $G$ be a chordal graph.
Then $C_{R}(G)$ is connected if and only if $G$ is connected.
\end{corollary}

\begin{proof}
Assume that $H$ and $H'$ are distinct connected components of $G$.
No maximal clique of $H$ can share a vertex with a maximal clique of $H'$.
It follows that there be no path of $C_{R}(G)$ that joins two such cliques.
Thus $C_{R}(G)$ is not connected.

The other direction is stated without proof in~\cite{HS12}*{p.~716}.
Assume that $G$ is connected.
Since $G$ is chordal it has a clique tree~\cite{Gavril}*{Theorem 2}, and Proposition~\ref{clique-tree-edges} shows that every edge of the clique tree is an edge of $C_{R}(G)$.
Thus $C_{R}(G)$ has a spanning tree, so it is connected.
\end{proof}

Next we note a characterisation of clique graphs due to Szwarcfiter and Bornstein.

\begin{theorem}[\cite{SB94}*{Theorem 2.1}]
\label{SB}
The graph $H$ is isomorphic to $C(G)$ for some connected chordal graph $G$ if and only if $H$ has a spanning tree $T$ such that whenever $u$ and $v$ are adjacent in $H$, the path of $T$ from $u$ to $v$ induces a clique of $H$.
\end{theorem}

\subsection{Induced cycles}
Next we observe that clique graphs can have induced cycles of any length.
We will later show that this is not true for reduced clique graphs.
For an integer $n\geq 3$ the \emph{wheel graph} with $n$ spokes is obtained from a cycle of $n$ vertices by adding a new vertex and making it adjacent to all vertices of the cycle.
Thus the wheel graph with $n$ spokes has an induced cycle of $n$ vertices.

\begin{proposition}
\label{clique-wheels}
For every integer $n\geq 3$ the wheel graph with $n$ spokes is isomorphic to the clique graph of a chordal graph.
\end{proposition}

\begin{proof}
This is easy to prove using Theorem~\ref{SB}, but we will give a direct construction.
Start with a clique on the $n+1$ vertices $u_{0},u_{1},\ldots, u_{n-1},x$.
For each $i\in \mathbb{Z}/n\mathbb{Z}$, add a new vertex $v_{i}$ and make it adjacent to $u_{i}$ and $u_{i+1}$.
Call the resulting graph $G$.
It is easy to verify that $G$ is chordal, and its maximal cliques are $\{u_{0},u_{1},\ldots, u_{n-1},x\}$ along with $\{v_{i},u_{i},u_{i+1}\}$ for each $i\in\mathbb{Z}/n\mathbb{Z}$.
The result follows.
\end{proof}

\begin{proposition}
\label{reduced-clique-wheels}
For every even integer $n\geq 4$ the wheel graph with $n$ spokes is isomorphic to the reduced clique graph of a chordal graph.
\end{proposition}

\begin{proof}
The example in Figure~\ref{Fig2} shows that this proposition is true when $n=4$.
Therefore we will assume that $n$ is an even integer exceeding $4$.
Let $k$ be $n/2$.
We start with a clique on $u_{0},u_{1},\ldots, u_{k-1},x$.
For each $i\in \mathbb{Z}/k\mathbb{Z}$, add a new vertex $v_{i}$ and make it adjacent to $u_{i}$ and $u_{i+1}$.
In addition, add a new vertex $z_{i}$ and make it adjacent to $u_{i}$.
Let $G$ be the resulting graph.
Certainly $G$ is chordal.
The maximal cliques of $G$ are $\{u_{0},u_{1},\ldots, u_{k-1},x\}$ along with $\{v_{i},u_{i},u_{i+1}\}$ and $\{u_{i},z_{i}\}$ for each $i\in\mathbb{Z}/k\mathbb{Z}$.
The fact that $k\geq 3$ means that $\{v_{i},u_{i},u_{i+1}\}$ and $\{v_{j},u_{j},u_{j+1}\}$ have at most one vertex in common when $i\ne j$.
We can easily verify that such a pair does not form a separating pair.
On the other hand, $\{u_{i},z_{i}\}$ forms a separating pair with both $\{v_{i},u_{i},u_{i+1}\}$ and $\{v_{i-1},u_{i-1},u_{i}\}$.
All maximal cliques form a separating pair with $\{u_{0},u_{1},\ldots, u_{k-1},x\}$.
In this way we can check that $C_{R}(G)$ is isomorphic to the wheel graph with $n$ spokes.
\end{proof}

If $G$ is the chordal graph $K_{1,4}$, then $G$ has four maximal cliques and $C_{R}(G)$ is the complete graph on four vertices (which is to say, the wheel with three spokes).
However, we see no way of constructing a reduced clique graph that is isomorphic to a wheel graph with a number of spokes that is odd and greater than three.
In fact, we are prepared to make the following, stronger, conjecture.

\begin{conjecture}
\label{odd-holes-conjecture}
Let $k>3$ be an odd integer.
There is no chordal graph $G$ such that $C_{R}(G)$ contains an induced cycle with exactly $k$ vertices.
\end{conjecture}

We prove the first case of this conjecture in the following work.

\begin{definition}
Let $G$ be a chordal graph.
Let $C_{0},C_{1},\ldots, C_{n-1}$ be a cyclic ordering of the maximal cliques in an induced cycle of $C_{R}(G)$.
We take the indices to be from $\mathbb{Z}/n\mathbb{Z}$, so $C_{i}$ and $C_{j}$ are adjacent in $C_{R}(G)$ if and only if $j\in \{i-1,i+1\}$.
If $|C_{i}\cap C_{i+1}|\leq |C_{j}\cap C_{j+1}|$ for every $j\in \mathbb{Z}/n\mathbb{Z}$, then we say that the edge between $C_{i}$ and $C_{i+1}$ is a \emph{minimal} edge of the cycle.
\end{definition}

\begin{lemma}
\label{holefree-lemma}
Let $G$ be a chordal graph.
Let $C_{0},C_{1},\ldots, C_{n-1}$ be a cyclic ordering of the maximal cliques in an induced cycle of $C_{R}(G)$, where $n\geq 4$ and the indices are from $\mathbb{Z}/n\mathbb{Z}$.
Assume that the edge between $C_{0}$ and $C_{1}$ is a minimal edge of the induced cycle.
Let $S$ be $C_{0}\cap C_{1}$ and for $i=0,1$ let $H_{i}$ be the connected component of $G-S$ that contains $C_{i}-S$.
Then $H_{0}$ and $H_{1}$ are distinct connected components and $C_{i}-S$ is contained in $H_{0}$ or $H_{1}$ for every $i\in \mathbb{Z}/n\mathbb{Z}$.
Furthermore, either:
\begin{enumerate}[label = \textup{(\roman*)}]
\item $H_{0}$ contains all of $C_{0}-S, C_{2}-S,\ldots, C_{n-1}-S$,
\item $H_{1}$ contains all of $C_{1}-S, C_{2}-S,\ldots, C_{n-1}-S$, or
\item $n=4$, and $H_{0}$ contains $C_{0}-S$ and $C_{2}-S$ while $H_{1}$ contains $C_{1}-S$ and $C_{3}-S$.
\end{enumerate}
\end{lemma}

\begin{proof}
Note that because $C_{0},C_{1},\ldots, C_{n-1}$ are distinct maximal cliques of $G$, none of them is contained in $S$.
Thus $C_{i}-S$ is non-empty for all $i$ and is contained in $H_{i}$, a connected component of $G-S$.
Because $C_{0}$ and $C_{1}$ form a separating pair, $C_{0}-S$ and $C_{1}-S$ are contained in different connected components of $G-S$, so $H_{0}$ and $H_{1}$ are distinct components.

\begin{claim}
\label{hole-free-claim1}
Assume that $i$ and $j$ are distinct indices in $\mathbb{Z}/n\mathbb{Z}$ such that $H_{i}$ and $H_{j}$ are distinct, and furthermore, $C_{i}$ is adjacent in $C_{R}(G)$ to $C_{p}$, where $C_{p}-S$ is not contained in $H_{i}$ and $C_{j}$ is adjacent to $C_{q}$, where $C_{q}-S$ is not contained in $H_{j}$.
Then $C_{i}$ and $C_{j}$ are adjacent in $C_{R}(G)$.
\end{claim}

\begin{proof}
Note that because the cycle of $C_{R}(G)$ is induced, $p$ is in $\{i-1,i+1\}$ and $q$ is in $\{j-1,j+1\}$.
Note also that $C_{i}\cap C_{p}$ is contained in $S$.
If this containment is proper then $|C_{i}\cap C_{p}| < |S| = |C_{0}\cap C_{1}|$ and we have violated our assumption that the edge between $C_{0}$ and $C_{1}$ is minimal.
Therefore $C_{i}$ and $C_{p}$ both contain $S$.
The same argument shows $S\subseteq C_{j}\cap C_{q}$.
Now $C_{i}\cap C_{j}$ is equal to $S$.
Moreover $C_{i}-S$ and $C_{j}-S$ are in different components of $G-S$, so $C_{i}$ and $C_{j}$ form a separating pair of maximal cliques.
Hence they are adjacent in $C_{R}(G)$.
\end{proof}

We colour the cliques of $C_{0},C_{1},\ldots, C_{n-1}$ in the following way.
For each $i\in \mathbb{Z}/n\mathbb{Z}$, if $C_{i}-S$ is contained in $H_{0}$ we colour $C_{i}$ red, and if $C_{j}-S$ is in $H_{1}$ we colour $C_{i}$ blue.
Thus $C_{0}$ is red and $C_{1}$ is blue.

\begin{claim}
\label{hole-free-claim2}
Any maximal clique $C_{i}$ is either red or blue.
\end{claim}

\begin{proof}
If the claim fails then there is some $i\in \mathbb{Z}/n\mathbb{Z}-\{0,1\}$ such that $H_{i}$ is not equal to $H_{0}$ or $H_{1}$.
We colour any clique $C_{j}$ in $C_{0},C_{1},\ldots, C_{n-1}$ green if $C_{j}-S$ is contained in $H_{i}$.
We know that the collections of red, blue, and green cliques are all non-empty.
Therefore we can find a red clique, $C_{\text{red}}$, adjacent to a clique that is not red.
We can similarly find $C_{\text{blue}}$, a blue clique that is adjacent to a non-blue clique, and $C_{\text{green}}$, a green clique that is adjacent to a clique that is not green.
Now Claim~\ref{hole-free-claim1} implies that $C_{\text{red}}$, $C_{\text{blue}}$, and $C_{\text{green}}$ are adjacent to each other in $C_{R}(G)$.
As they are three distinct vertices in an induced cycle of $C_{R}(G)$ with at least four vertices, this is an immediate contradiction.
\end{proof}

If $C_{1}$ is the only blue clique, then statement (i) holds and we have nothing left to prove.
Similarly, if $C_{0}$ is the only red clique, then (ii) holds and we are done.
So we assume there are at least two red cliques and at least two blue cliques.
We can choose $C_{\text{red}}$ and $C_{\text{red}}'$ to be distinct red cliques that are adjacent to blue cliques, and we can choose $C_{\text{blue}}$ and $C_{\text{blue}}'$ to be two distinct blue cliques that are adjacent to red cliques.
Now Claim~\ref{hole-free-claim1} implies that $C_{\text{red}}$ and $C_{\text{red}}'$ are adjacent to both $C_{\text{blue}}$ and $C_{\text{blue}}'$.
Thus the four cliques induce a cycle in $C_{R}(G)$.
This is impossible if $n\geq 5$, so we conclude that $n=4$.
Now $C_{0}$ is a red clique and it is adjacent to two blue cliques.
Thus $C_{1}$ and $C_{3}$ are blue, $C_{2}$ is red, and we are finished.
\end{proof}

We next establish the first case of Conjecture~\ref{odd-holes-conjecture}.

\begin{lemma}
\label{no-five-holes}
There is no chordal graph $G$ such that $C_{R}(G)$ has an induced cycle with exactly five vertices.
\end{lemma}

\begin{proof}
Assume otherwise and let $G$ be a chordal graph such that $C_{R}(G)$ contains an induced cycle with five vertices.
Let $C_{0},C_{1},C_{2}, C_{3}, C_{4}$ be the maximal cliques in this cycle, where the indices are from $\mathbb{Z}/5\mathbb{Z}$ and $C_{i}$ is adjacent to $C_{j}$ if and only if $j\in\{i-1,i+1\}$.
By adding a constant to these indices as necessary, we may assume that
\[|C_{0}\cap C_{1}|\leq |C_{i}\cap C_{i+1}|\]
for all $i\in \mathbb{Z}/5\mathbb{Z}$, so that the edge between $C_{0}$ and $C_{1}$ is a minimal edge of the cycle.
Let $S$ be $C_{0}\cap C_{1}$.
Note that $S$ is non-empty.

Now we apply Lemma~\ref{holefree-lemma}.
By applying the permutation $\rho\colon i\mapsto 1-i$ as necessary, we may assume that statement (ii) in Lemma~\ref{holefree-lemma} applies.
Therefore we let $H_{0}$ and $H_{1}$ be connected components of $G-S$ such that $H_{0}$ contains $C_{0}-S$ and $H_{1}$ contains $C_{1}-S$, $C_{2}-S$, $C_{3}-S$, and $C_{4}-S$.

\begin{claim}
\label{five-hole-1}
$C_{0}\cap C_{4}=S=C_{0}\cap C_{1}$.
\end{claim}

\begin{proof}
Because $C_{0}-S$ and $C_{4}-S$ are contained in different components of $G-S$, it follows that $C_{0}\cap C_{4}\subseteq S$.
All we have left to prove is that this containment is not proper.
If it were proper, then we would contradict the assumption that the edge between $C_{0}$ and $C_{1}$ is minimal.
\end{proof}

\begin{claim}
\label{five-hole-2}
Neither $C_{2}$ nor $C_{3}$ contains $S$.
\end{claim}

\begin{proof}
Note that $C_{0}\cap C_{2}\subseteq S$ because $C_{0}-S$ and $C_{2}-S$ are contained in different components of $G-S$.
Certainly any path from a vertex of $C_{0}-C_{2}$ to a vertex of $C_{2}-C_{0}$ must use a vertex of $S$.
If $C_{0}\cap C_{2} = S$, then $C_{0}$ and $C_{2}$ form a separating pair, so $C_{0}$ and $C_{2}$ are adjacent in $C_{R}(G)$.
This contradicts the fact that $C_{0}$ and $C_{2}$ are non-consecutive vertices in an induced cycle.
The same argument shows that $C_{3}$ does not contain $S$.
\end{proof}

\begin{claim}
\label{five-hole-3}
$C_{2}\cap C_{4}\subseteq C_{1}$ and $C_{3}\cap C_{1}\subseteq C_{4}$.
\end{claim}

\begin{proof}
Assume that $x$ is a vertex of $C_{2}\cap C_{4}$ that is not in $C_{1}$.
By Claim~\ref{five-hole-2} we can let $y$ be a vertex in $S-C_{2}$.
Thus $y$ is in $C_{1}-C_{2}$.
So $x$ is in $C_{2}-C_{1}$ and $y$ is in $C_{1}-C_{2}$.
Claim~\ref{five-hole-1} implies that $y$ is in $C_{4}$.
As $x$ is also in $C_{4}$ we see that $x$ and $y$ are adjacent.
Because $C_{1}$ and $C_{2}$ are adjacent in $C_{R}(G)$ they have a non-empty intersection, but now the edge $xy$ shows that $C_{1}$ and $C_{2}$ do not form a separating pair and we have a contradiction.
A symmetric argument shows $C_{3}\cap C_{1}\subseteq C_{4}$.
\end{proof}

\begin{claim}
\label{five-hole-4}
$C_{2}$ contains a vertex of $C_{1}-C_{4}$ and $C_{3}$ contains a vertex of $C_{4}-C_{1}$.
\end{claim}

\begin{proof}
By symmetry it suffices to prove the first statement.
Assume that $C_{2}$ contains no vertex of $C_{1}-C_{4}$.
Because $C_{1}$ and $C_{2}$ are adjacent in $C_{R}(G)$, they have at least one vertex in common.
By our assumption, no vertex of $C_{1}\cap C_{2}$ is in $C_{1}-C_{4}$, so any such vertex must be in $C_{1}\cap C_{4}$.
Therefore $C_{2}$ and $C_{4}$ are not disjoint.
Since $C_{2}$ and $C_{4}$ are not adjacent in $C_{R}(G)$, we can let $P$ be a $(C_{2}\cap C_{4})$\dash avoiding path from a vertex $x\in C_{2}-C_{4}$ to $y\in C_{4}-C_{2}$.

Our assumption means that $x$ is not in $C_{1}$, so it is in $C_{2}-C_{1}$.
Our assumption and Claim~\ref{five-hole-3} imply that $C_{2}\cap C_{4}=C_{2}\cap C_{1}$.
Therefore $P$ is a $(C_{2}\cap C_{1})$\dash avoiding path.
But Claim~\ref{five-hole-2} shows that we can choose a vertex $z$ in $S-C_{2}$.
Thus $z$ is in $C_{1}-C_{2}$ and Claim~\ref{five-hole-1} shows that $z$ is in $C_{4}$.
Assuming that $z$ and $y$ are not equal, they are adjacent, as both are in $C_{4}$.
By appending (if necessary) the edge $yz$ to the end of $P$ we obtain a $(C_{1}\cap C_{2})$\dash avoiding path from a vertex in $C_{2}-C_{1}$ to a vertex in $C_{1}-C_{2}$.
Hence $C_{1}$ and $C_{2}$ do not form a separating pair and this contradicts the fact that they are adjacent in $C_{R}(G)$.
\end{proof}

\begin{claim}
\label{five-hole-5}
Either $C_{2}\cap (C_{1}\cap C_{4}) \subseteq C_{3}$ or $C_{3}\cap (C_{1}\cap C_{4})\subseteq C_{2}$.
\end{claim}

\begin{proof}
Note that $C_{2}\cap C_{3}$ is non-empty, since $C_{2}$ and $C_{3}$ are adjacent in $C_{R}(G)$.
If the claim fails, then we choose $x\in (C_{2}\cap C_{1}\cap C_{4})-C_{3}$ and $y\in (C_{3}\cap C_{1}\cap C_{4})-C_{2}$.
Now $x$ and $y$ are both in $C_{1}\cap C_{4}$, so they are adjacent.
Moreover $x$ is in $C_{2}-C_{3}$ and $y$ is in $C_{3}-C_{2}$.
Thus $C_{2}$ and $C_{3}$ do not form a separating pair and we have a contradiction.
\end{proof}

By using Claim~\ref{five-hole-5}, we will assume that $C_{2}\cap (C_{1}\cap C_{4})$ is a subset of $C_{3}$.
The other outcome from Claim~\ref{five-hole-5} yields to a symmetric argument.
Using Claim~\ref{five-hole-2} we choose a vertex $x\in S$ that is not in $C_{3}$.
Note that Claim~\ref{five-hole-1} implies that $S$ is contained in $C_{1}\cap C_{4}$.
So $x$ is in $(C_{1}\cap C_{4})-C_{3}$.
The assumption $C_{2}\cap (C_{1}\cap C_{4})\subseteq C_{3}$ implies that $x$ is not in $C_{2}$.

By Claim~\ref{five-hole-4} we can also choose $y$ in $C_{2}\cap (C_{1}-C_{4})$ and $z$ in $C_{3}\cap (C_{4}-C_{1})$.
Because $y$ is in $C_{1}-C_{4}$ and 
Claim~\ref{five-hole-3} says that $C_{3}\cap C_{1}$ is a subset of $C_{4}$, it follows that $y$ is not in $C_{3}$.
Therefore $y$ belongs to $C_{2}-C_{3}$.
A symmetric argument shows that $z$ is in $C_{3}-C_{2}$.
Now $x$ and $y$ are adjacent as they are both in $C_{1}$, and $x$ and $z$ are adjacent as they are both in $C_{4}$.
Note that $C_{2}\cap C_{3}$ is non-empty as $C_{2}$ and $C_{3}$ are adjacent in $C_{R}(G)$.
But the path with vertex sequence $y,x,z$ is $(C_{2}\cap C_{3})$\dash avoiding, so $C_{2}$ and $C_{3}$ do not form a separating pair.
This final contradiction completes the proof.
\end{proof}

Lemma~\ref{no-five-holes} shows that the class of reduced clique graphs is contained in the class of graphs with no length-five induced cycle.
We next show that this containment is proper.

\begin{proposition}
Let $n\geq 4$ be an integer.
There is no chordal graph $G$ such that either $C(G)$ or $C_{R}(G)$ is a cycle with $n$ vertices.
\end{proposition}

\begin{proof}
Let $H$ be a cycle with at least four vertices.
Any spanning tree of $H$ is a Hamiltonian path.
The end vertices of this path are adjacent in $H$, but the path of the spanning tree between these vertices does not induce a clique in $H$.
Therefore $H$ is not isomorphic to $C(G)$ for any chordal graph $G$ by Theorem~\ref{SB}.

We turn to reduced chordal graphs.
Assume for a contradiction that $G$ is a chordal graph with $C_{0},C_{1},\ldots, C_{n-1}$ as its list of maximal cliques, where the indices are from $\mathbb{Z}/n\mathbb{Z}$ for some $n\geq 4$, and $C_{i}$ is adjacent to $C_{j}$ in $C_{R}(G)$ if and only if $j\in\{i-1,i+1\}$.
We can assume without loss of generality that the edge between $C_{0}$ and $C_{1}$ is a minimal edge of $C_{R}(G)$.
Let $S$ be $C_{0}\cap C_{1}$.
Assume that statement (iii) in Lemma~\ref{holefree-lemma} holds.
Thus $n=4$ and there are distinct connected components, $H_{0}$ and $H_{1}$, of $G-S$ such that $H_{0}$ contains $C_{0}-S$ and $C_{2}-S$ while $H_{1}$ contains $C_{1}-S$ and $C_{3}-S$.
Note that $C_{0}\cap C_{3}\subseteq S$, and in fact $C_{0}\cap C_{3}$ is equal to $S$, or else the minimality of the $C_{0}$-$C_{1}$ edge is contradicted.

Either  $C_{0}\cap C_{2}$ is empty, or it is not.
In the latter case, we can apply Proposition~\ref{CRG-path} to $C_{0}$ and $C_{2}$.
We see that either $C_{0}\cap C_{1}$ or $C_{0}\cap C_{3}$ properly contains $C_{0}\cap C_{2}$.
By symmetry, we can assume $C_{0}\cap C_{2}$ is a proper subset of $C_{0}\cap C_{1}=S$.
Thus $C_{0}-S$ and $C_{2}-S$ are disjoint sets.
We can let $P$ be a shortest-possible path of $H_{0}$ from a vertex of $C_{0}-S$ to a vertex of $C_{2}-S$.
On the other hand, if $C_{0}\cap C_{2}$ is empty, then $C_{0}-S$ and $C_{2}-S$ are again disjoint subsets in $H_{0}$, so we again let $P$ be a shortest-possible path of $H_{0}$ from $C_{0}-S$ to $C_{2}-S$.
In either case, $P$ contains exactly one vertex of $C_{0}$ and exactly one vertex of $C_{2}$.
Then $P$ must contain at least one edge, and this edge is in a maximal clique that is equal to neither $C_{0}$ nor $C_{2}$.
Nor can this maximal clique be $C_{1}$ or $C_{3}$, because any edge of $P$ is contained in $H_{0}$.
So we have a contradiction in the case that 
(iii) in Lemma~\ref{holefree-lemma} holds.

Now we assume that either (i) or (ii) holds.
By applying the permutation $\rho\colon i\mapsto 1-i$ as necessary, we will assume that $H_{0}$ and $H_{1}$ are distinct connected components of $G-S$, and that $H_{0}$ contains $C_{0}-S$ while $H_{1}$ contains $C_{i}-S$ for $i=\{1,2,\ldots, n-1\}$.
By the same argument as earlier, we can see that $C_{n-1}$ contains $S$, or else the choice of the $C_{0}-C_{1}$ edge is contradicted.

Now $C_{1}\cap C_{n-1}$ contains $S$, and $C_{1}$ and $C_{n-1}$ are non-adjacent in $C_{R}(G)$.
We apply Proposition~\ref{CRG-path} and see that there is a path of $C_{R}(G)$ from $C_{1}$ to $C_{n-1}$ such that every intersection of consecutive cliques in the path properly contains $C_{1}\cap C_{n-1}$.
This path is either $C_{1},C_{0},C_{n-1}$, or it is $C_{1},C_{2},\ldots, C_{n-1}$.
Assume the former.
Then $C_{1}\cap C_{0}=S$ properly contains $C_{1}\cap C_{n-1}\supseteq S$ and we have a contradiction.
Hence any intersection of consecutive cliques in $C_{1},C_{2},\ldots, C_{n-1}$ properly contains $C_{1}\cap C_{n-1}$, and hence contains $S$.
It follows that $C_{2}$ contains $S$ and thus $C_{0}\cap C_{2}$ is non-empty.

Since $C_{0}-S$ and $C_{2}-S$ are contained in different components of $G-S$, any path of a vertex from $C_{0}-C_{2}$ to a vertex of $C_{2}-C_{0}$ must contain a vertex of $S=C_{0}\cap C_{2}$.
Thus $C_{0}$ and $C_{2}$ form a separating pair in $G$, and hence they are adjacent in $C_{R}(G)$, which is a contradiction.
\end{proof}

\subsection{Clique graphs vs.\ reduced clique graphs}
Consider the classes $\{C(G)\}$ and $\{C_{R}(G)\}$, where $G$ ranges over all chordal graphs.
Proposition~\ref{clique-wheels} and Lemma~\ref{no-five-holes} show that the wheel with five spokes is isomorphic to a graph in the former class but not the latter.
Is there a graph that is isomorphic to a graph in the latter class but not the former?
We will show that the answer is, once again, yes.
Recall that if $G$ and $G'$ are disjoint graphs, then $G\boxtimes G'$ is obtained from the union of $G$ and $G'$ by making every vertex of $G$ adjacent to every vertex of $G'$.
We use $P_{n}$ to denote the path with $n$ vertices.

\begin{lemma}
\label{path-products}
Let $m,n\geq 1$ be integers.
Then $P_{m}\boxtimes P_{n}$ is isomorphic to the reduced clique graph of a chordal graph.
If $n\geq 22$, then $P_{n}\boxtimes P_{n}$ is not isomorphic to the clique graph of a chordal graph.
\end{lemma}

\begin{proof}
Let $G$ be the graph obtained from the disjoint union of $P_{m+1}$ and $P_{n+1}$ by adding a new vertex that is adjacent to every vertex of the disjoint union.
It is easy to confirm that $G$ is chordal, and that $C_{R}(G)$ is isomorphic to $P_{m}\boxtimes P_{n}$.

For the second statement, we let $H$ be a graph with disjoint induced paths $P_{u}=u_{0},u_{1},\ldots, u_{n-1}$ and $P_{v}=v_{0},v_{1},\ldots, v_{n-1}$, where $n\geq 22$ and every $u_{i}$ is adjacent to every $v_{j}$.
Thus $H$ is isomorphic to $P_{n}\boxtimes P_{n}$.
We will assume for a contradiction that $H$ is isomorphic to $C(G)$ for some chordal graph $G$.
Because $C(G)$ is connected it follows easily that $G$ is connected, so we can apply Theorem~\ref{SB} and deduce that $H$ has a spanning tree $T$, where the path of $T$ from $u$ to $v$ induces a clique of $H$ whenever $u$ and $v$ are adjacent in $H$.

\begin{claim}
\label{path-product-1}
Let $i$ and $j$ be integers satisfying $0<i,j<n-1$.
The path of $T$ from $u_{i}$ to $v_{j}$ is contained in one of:
$\{u_{i},u_{i+1},v_{j},v_{j+1}\}$,
$\{u_{i},u_{i+1},v_{j-1},v_{j}\}$,
$\{u_{i-1},u_{i},v_{j},v_{j+1}\}$,
$\{u_{i-1},u_{i},v_{j-1},v_{j}\}$.
\end{claim}

\begin{proof}
Let $P$ be the path of $T$ from $u_{i}$ to $v_{j}$.
Since $u_{i}$ is adjacent to $v_{j}$ it follows that $P$ induces a clique of $H$.
As $u_{i}$ is not adjacent to any of the vertices in $u_{0},\ldots, u_{i-2},u_{i+2},\ldots, u_{n-1}$, it follows that the vertices of $P$ that are in $P_{u}$ belong to $\{u_{i-1},u_{i},u_{i+1}\}$.
Similarly, the vertices of $P$ that are in $P_{v}$ belong to $\{v_{j-1},v_{j},v_{j+1}\}$.
But $u_{i-1}$ is not adjacent to $u_{i+1}$, so $P$ does not contain both.
The claim follows by symmetry.
\end{proof}

Claim~\ref{path-product-1} implies that the path of $T$ between $u_{i}$ and $v_{j}$ has at most three edges.

Let $P$ be a longest-possible path of $T$ and let $p_{0},p_{1},\ldots, p_{k-1}$ be the vertices of $P$.
For $i=0,1,\ldots, k-1$, let $U_{i}$ be the set of vertices in $P_{u}$ such that $u$ is in $U_{i}$ if and only if the shortest path of $T$ from $u$ to a vertex in $P$ contains $p_{i}$.
We define $V_{i}$ to be the analogous set of vertices in $P_{v}$.
Note that $(U_{0},U_{1},\ldots, U_{k-1})$ is a partition of the vertices of $P_{u}$, and $(V_{0},V_{1},\ldots, V_{k-1})$ is a partition of the vertices of $P_{v}$.

\begin{claim}
\label{path-product-2}
Either
\[\max\{|U_{i}|\colon 0\leq i\leq k-1\}\leq 3\quad \text{or}\quad \max\{|V_{i}|\colon 0\leq i\leq k-1\}\leq 3.\]
\end{claim}

\begin{proof}
Assume for a contradiction that $|U_{i}|\geq 4$ and $|V_{j}|\geq 4$.
Let $p$ and $q$, respectively, be the smallest (largest) integers such that $u_{p}, u_{q}\in U_{i}$.
Then $q-1 > p+1$ because $|U_{i}|\geq 4$.
In the same way, let $s$ and $t$ be the smallest (largest) integers such that $v_{s},v_{t}\in V_{j}$.
Then $t -1> s+1$.
It is simple to see from Claim~\ref{path-product-1} that the path of $T$ from $u_{p}$ to $v_{s}$ has no vertex in common with the path of $T$ from $u_{q}$ to $v_{t}$.
But this contradicts the fact that both paths contain $p_{i}$ and $p_{j}$.
\end{proof}

By using Claim~\ref{path-product-2}, we will assume without loss of generality that $|V_{i}|\leq 3$ for each $i=0,1,\ldots, k-1$.
Since $(U_{0},U_{1},\ldots, U_{k-1})$ is a partition of the vertices in $P_{u}$ we can choose $i$ so that $U_{i}$ contains a vertex $x$.
We claim that if $j\leq i-4$ or $j\geq i+4$, then $V_{j}=\emptyset$.
If this fails, then the path of $T$ from a vertex in $V_{j}$ to $x$ contains at least four edges of $P$.
But this contradicts our earlier conclusion that any path of $T$ from a vertex of $P_{u}$ to a vertex of $P_{v}$ contains at most three edges.
So now the vertices of $P_{v}$ belong to
\[
V_{i-3}\cup V_{i-2}\cup \cdots \cup V_{i+2}\cup V_{i+3}
\]
and this union has cardinality at most $7\times 3$.
Thus $P_{v}$ contains at most $21$ vertices and this contradicts $n\geq 22$.
\end{proof}

\section{Conclusions and open problems}

We restate Conjecture~\ref{odd-holes-conjecture} here.

\begin{conjecture}
Let $k>3$ be an odd integer.
There is no chordal graph $G$ such that $C_{R}(G)$ contains an induced cycle with exactly $k$ vertices.
\end{conjecture}

So far as we have been able to tell, every chordal graph is isomorphic to both a clique graph, and to a reduced clique graph.
We conjecture this holds generally.

\begin{conjecture}
Let $H$ be a chordal graph.
There are chordal graphs $G$ and $G'$ such that $H$ is isomorphic to both $C(G)$ and $C_{R}(G')$.
\end{conjecture}

Szwarcfiter and Bornstein present a polynomial-time algorithm for deciding whether a given graph is isomorphic to $C(G)$ for some chordal graph $G$~\cite{SB94}.
Their techniques do not obviously extend to recognising reduced clique graphs.
Nonetheless, we will make the following conjecture.

\begin{conjecture}
There is a polynomial-time algorithm for deciding whether a given graph is isomorphic to $C_{R}(G)$ for some chordal graph $G$.
\end{conjecture}

More informally, we ask if there is a structural description for reduced clique graphs that is analogous to Theorem~\ref{SB}.

\section{Acknowledgments}

We thank the anonymous referee for their careful work, and for providing the construction in Proposition~\ref{reduced-clique-wheels}.

\section{Statements \& declarations}

The first author was supported by a Rutherford Discovery Fellowship, managed by Royal Society Te Ap\={a}rangi.

The authors have no relevant financial or non-financial interests to disclose.

All authors contributed to the study conception and design.
Much of the material appeared previously in the PhD thesis of the second author.
All authors read and approved the final manuscript.

\section{Data availability statement}

Data sharing not applicable to this article as no datasets were generated or analysed during the current study.

\end{document}